\documentclass[11pt]{article}

\usepackage{graphicx}

\usepackage{amsmath}
\usepackage{amsfonts}
\usepackage{amssymb}
\usepackage{amsbsy} 

\usepackage{amsthm}
\newtheorem{thm}{Theorem}
\newtheorem{prop}{Proposition}
\newtheorem{lem}{Lemma}
\newtheorem{cor}{Corollary}
\theoremstyle{definition} \newtheorem{de}{Definition}
\theoremstyle{plain} \newtheorem*{claim}{Claim}
\theoremstyle{plain} \newtheorem*{remark}{Remark}
\theoremstyle{definition} \newtheorem{ex}{Example}

\DeclareMathOperator{\core}{core}

\def\bs#1{\boldsymbol#1}
\def\cg#1{\core (#1)}

\def\img{i}

\begin{document}

\title{New graph polynomials from the Bethe approximation of 
the Ising partition function}
\author{
  \\[-1mm]
  {Yusuke Watanabe}\footnote{The Institute of Statistical Mathematics, 10-3 Midori-cho, Tachikawa, Tokyo, Japan}    \\
  {\tt watay@ism.ac.jp}    \\
  \and
  \\[-1mm]
  {Kenji Fukumizu$*$}\\    
  {\tt fukumizu@ism.ac.jp} 
}
\date{}
\maketitle
\thispagestyle{empty}   


\begin{abstract}
We introduce two graph polynomials and discuss their properties. One
is a polynomial of two variables whose investigation is motivated by the performance analysis of
the Bethe approximation of the Ising partition function. The other is a
polynomial of one variable that is obtained by the specialization of the first one. 
It is shown that these polynomials satisfy deletion-contraction relations and
are new examples of the V-function, which was introduced by Tutte
(1947, Proc. Cambridge Philos. Soc. 43, 26-40). For these polynomials,
we discuss the interpretations of special values and then obtain the bound
on the number of sub-coregraphs, i.e., spanning subgraphs with no
vertices of degree one. It is proved that the polynomial of one
variable is equal to the monomer-dimer partition function with weights
parameterized by that variable. The properties of the coefficients and the
possible region of zeros are also discussed for this polynomial.
\end{abstract}


\section{Introduction and terminologies}
\subsection{Introduction}
The aim of this paper is to introduce two new graph polynomials
and study their properties.
The first one is a two-variable polynomial denoted by
$\theta_{G}(\beta,\gamma)$ 
and the second one is a one-variable polynomial 
denoted by $\omega_{G}(\beta)$, which is obtained as a specialization of 
$\theta_{G}$.

Partition functions studied in statistical physics have been a source of 
various graph polynomials. 
For example, the partition functions of
the q-state Potts model and 
the bivariated random-cluster model of Fortuin and Kasteleyn
provide graph polynomials.
They are known to be equivalent to the Tutte polynomial
\cite{Bollobas}.
Another example is
the monomer-dimer partition function with uniform 
monomer and dimer weights,
which is essentially the matching polynomial \cite{HLmonomerdimer}.

The polynomial $\theta_{G}$
comes from the problem of computing the Ising partition function
defined as
\begin{equation}
Z(G;\bs{J},\bs{h}):= \sum_{x_1,\ldots,x_N=\pm 1} 
\exp
\Big(
{\sum_{e \in E \atop e=ij} J_e x_i x_j + \sum_{i \in V} h_i x_i }
\Big),  \label{defIsing}
\end{equation} 
where 
$J_e$ and $h_i$ denote coupling constants
and local external fields, respectively, and
 $G=(V,E)$ is the underlying graph.
In general, the partition function is computationally intractable and
the Bethe approximation is a popular method for its approximation \cite{Bethe}.
The approximation ratio, which evaluates the performance 
of this method, depends on the structure of the underlying graph.
In particular, if the graph is a tree, 
the ratio is equal to one, i.e.,
the Bethe approximation gives the exact value of
the partition function.
In principle, the approximation becomes more difficult
as the nullity increases.
In \cite{WFloop},
it is shown that the ratio
is described by a multivariate polynomial
$\Theta_{G}(\bs{\beta},\bs{\gamma})$. 
We derive
the graph polynomial $\theta_{G}(\beta,\gamma)$
as its two-variable version.

The polynomial $\omega_{G}(\beta)$ is obtained from
$\theta_{G}(\beta,\gamma)$
by specializing $\gamma=2 \sqrt{-1}$ and 
eliminating a factor $(1-\beta)^{|E|-|V|}$.
We show that the polynomial coincides with 
the monomer-dimer partition function 
with weights parametrized by $\beta$.
In particular, for regular graphs,
$\omega-$polynomials are equal to the matching polynomials up to transformations.

We discuss the properties of $\theta_{G}$ and $\omega_{G}$ from the 
viewpoint of graph polynomials.
The most important feature of these graph polynomials is the
deletion-contraction relation:
\begin{align}
&\theta_{G}({\beta},\gamma)=
(1-\beta) \theta_{G\backslash e}({\beta},\gamma) +
\beta  \theta_{G/e}({\beta},\gamma),   \nonumber \\
&\omega_{G}({\beta})=
\omega_{G\backslash e}({\beta}) +
\beta  \omega_{G/e}({\beta}),   \nonumber
\end{align}
holds whenever $e \in E$ is not a loop.
Note that the graph $G\backslash e$ is obtained from $G$ by the deletion of the
edge $e$, and the graph $G/e$ is the result of the contraction of $e$.
Furthermore, these polynomials are multiplicative:
\begin{equation}
 \theta_{G_1 \cup G_2} =  \theta_{G_1}  \theta_{G_2}  \quad \text{ and } \quad 
 \omega_{G_1 \cup G_2} =  \omega_{G_1}  \omega_{G_2}, \nonumber
\end{equation}
where $G_1 \cup G_2$ is the disjoint union of $G_1$ and $G_2$.
Graph invariants that satisfy 
the deletion-contraction relation and the multiplicative law
have been studied by Tutte \cite{Tring} as the V-function.
Our graph polynomials $\theta_{G}$ and $\omega_{G}$ are essentially
examples of V-functions.

Graph polynomials that satisfy deletion-contraction relations
arise from a wide range of problems
\cite{Bollobas,EMinter1}. 
Most of them are known to be equivalent to the Tutte polynomial
or to be obtained by its specialization, and thus, they have reduction formulae even for loops.
Our new graph polynomials do not have such reduction formulae for loops and
are essentially different from the Tutte polynomial.

There have been few researches on specific V-functions except for those on the Tutte polynomial.
The Tutte polynomial has attracted interest because of its rich mathematical properties
such as matroid invariance and connections to links \cite{Welsh,Bollobas}.
These properties are not shared by general V-functions.
As described in this paper, our new V-functions also possess special properties,
and thus, their investigation should be fruitful.

The remainder of this paper is organized as follows.
In Section \ref{sec1.2},
the definitions and notations on graphs are described. 
Sections \ref{secTheta}, \ref{secMotivation}, and \ref{secfurther}
deal with the investigation of the $\theta$-polynomial:
the definition and basic properties of the $\theta$-polynomial are
presented in Section \ref{secTheta},
the motivation for the definition is
presented in Section \ref{secMotivation}, and
the special values of $\theta_{G}$ are discussed in Section \ref{secfurther}.
Section \ref{secomega} deals with the investigation of $\omega_{G}$
including a study on the special value, $\beta=1$.

\subsection{Basic terminologies and definitions}\label{sec1.2}
Let $G=(V,E)$ be a finite graph, where
$V$ is the set of vertices and 
$E$, the set of undirected edges.
In this paper, a graph implies a multigraph, in which
loops and multiple edges are allowed.
A subset $s$ of $E$ is identified with the spanning subgraph
$(V,s)$ of $G$ unless otherwise stated.

The notation $e=ij$ is used to indicate that 
vertices $i$ and $j$ are the endpoints of $e$. 
The number of ends of edges connecting to a vertex $i$ is called 
the {\it degree} of $i$ and denoted by $d_i$.

The number of connected components of $G$ is denoted by $k(G)$. 
The {\it nullity} and the {\it rank} of $G$ are defined by
$n(G):=|E|-|V|+k(G)$ and $r(G):=|V|-k(G)$, respectively.

For an edge $e \in E$,
the graph
$G \backslash e$ is obtained by deleting $e$
and $G/e$ is obtained by contracting $e$. 
If $e$ is a loop, $G/e$ is the same as $G \backslash e$.
The disjoint union of graphs $G_1$ and
$G_2$ is denoted by $G_1 \cup G_2$.
The graph with a single vertex and
$n$ loops is called the {\it bouquet graph} and denoted by $B_n$. 

The {\it core} of a graph $G$ 
is obtained by a process of clipping
vertices of degree one step-by-step \cite{Stopology}. 
This graph is denoted by $\cg{G}$.
For example, the core of a forest $F$ is the graph of $k(F)$ vertices
without edges.
A graph $G$ is called a {\it coregraph} if $G=\cg{G}$. 
In other words, a graph is a coregraph if and only if the degree of each
vertex is not equal to one.  
Note that the core of a graph is also called the {\it 2-core} \cite{PWinside}
and can be generalized to the notion of the {\it k-core} \cite{Bevolution,Rk-core}.

\section{Two-variable graph polynomial $\theta$} \label{secTheta}

\subsection{Definition}
First, we present the definition of a graph polynomial.
For the definition, we define a set of polynomials $\{f_n(x)\}_{n=0}^{\infty}$ 
inductively by the relations
\begin{equation}
 f_0(x)=1, \quad
f_1(x)=0,   \quad
\text{ and } \quad
f_{n+1}(x)=x f_n(x) + f_{n-1}(x). \label{defindf}
\end{equation}
Therefore, for instance, $f_2(x)=1,f_3(x)=x$, and so on.
Note that
these polynomials are transformations of the Chebyshev polynomials of the second
kind:
$f_{n+2}(2 \sqrt{-1} z)=(\sqrt{-1})^{n} U_{n}(z)$,
where $U_{n}(\cos \theta)= \frac{\sin ((n+1)\theta)}{\sin \theta}$.

\begin{de}
\label{deftheta}
For a given graph $G$, 
\begin{equation}
\theta_{G}({\beta},{\gamma}):= 
\sum_{s \subset E}
\beta^{|s|}
\prod_{i \in V} f_{d_i(s)}(\gamma)
\quad
\in \mathbb{Z} [ \beta,\gamma ],  \label{defthetaeq}
\end{equation}
where $d_{i}(s)$ is the degree of the vertex $i$ in $s$.
\end{de}

In Eq.~(\ref{defthetaeq}), there exists a summation over all subsets of $E$.
Recall that an edge set $s$ is identified with the spanning subgraph $(V,s)$.
Since $f_1(x)=0$, 
the subgraph $s$ contributes to the summation only if $s$ does
not have a vertex of degree one. 
Therefore, the summation is regarded as the summation over all
coregraphs of the forms $(V,s)$;
we call these {\it sub-coregraphs}.
In relevant papers, such subgraphs are called generalized loops
\cite{CCloopPRE,CCloop} or closed subgraphs \cite{Nseries,Nsubgraph}.

The following facts are immediate from the definition.
\begin{prop}$\quad$ 
\label{propbasictheta}
\begin{itemize}
\item[\rm{(a)}] \ \
$\theta_{G_1 \cup G_2}({\beta},{\gamma})
 =\theta_{G_1}({\beta},{\gamma})\theta_{G_2}({\beta},{\gamma}). $
\item[\rm{(b)}] \ \
$\theta_{B_n}({\beta},\gamma)=
\sum_{k=0}^{n} {n \choose k} f_{2k}(\gamma) \beta^{k}.$ 
\item[\rm{(c)}] \ \
$ \theta_{G}({\beta},{\gamma})=\theta_{\cg{G}}({\beta},{\gamma}).$
\end{itemize}
\end{prop}

\begin{ex}
For a tree $T$, $\theta_{T}(\beta,\gamma)=1$.
For a cycle graph $C_n$ having $n$ vertices and $n$ edges,
$\theta_{C_n}(\beta,\gamma)=1+\beta^{n}$. 
For the complete graph $K_{4}$,
$\theta_{K_4}(\beta,\gamma)=1+4\beta^{3}+3 \beta^{4}+6
 \beta^{5}\gamma^{2}+\beta^{6}\gamma^{4}$.
For the graph $X_1$, as shown in Figure \ref{X1X2},
$\theta_{X_1}(\beta,\gamma)=1+3 \beta^{2}+\beta^{3}\gamma^{2}$.
For the graph $X_2$, as also shown in Figure \ref{X1X2},
$\theta_{X_2}(\beta,\gamma)=1+2\beta+\beta^{2}+\beta^{3}\gamma^{2}$.
\end{ex}

\begin{figure}
\begin{center}
\includegraphics[scale=0.3]{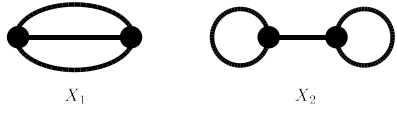} 
\vspace{-1mm}
\caption{Graph $X_1$ and $X_2$}
\label{X1X2}
\end{center}
\end{figure}

\subsection{Deletion-contraction relation and expression as Tutte's V-function}
\subsubsection{Deletion-contraction relation}
We prove the most important property of the
graph polynomial, $\theta$, called the deletion-contraction relation. 
The following formula of $f_n(x)$ plays an important role in the proof of this relation.
\begin{lem}\label{lemf}
 ${}^\forall n,m \in \mathbb{N}$,
\begin{equation}
f_{n+m-2}(x)= f_n(x)f_m(x)+f_{n-1}(x)f_{m-1}(x). \nonumber
\end{equation}
\end{lem}
\begin{proof}
Easily proved by induction using Eq.(\ref{defindf}).
\end{proof}

\begin{thm}
[Deletion-contraction relation]
\label{thmcd}
For a non-loop edge $e \in E$,
\begin{equation}
 \theta_{G}(\beta,\gamma)=
(1-\beta) \theta_{G\backslash e}(\beta,\gamma) +
\beta  \theta_{G/e}({\beta},\gamma).  \nonumber
\end{equation}
\end{thm}
\begin{proof}
Classify subgraph $s$ in the sum of Eq.~(\ref{defthetaeq}) depending on whether $s$ includes $e$ or not.
The former subgraph $s  \ni e=ij$ yields 
$-\beta \theta_{G \backslash e}+\beta\theta_{G / e}$,
where Lemma \ref{lemf} is used with $n=d_i$ and $m=d_j$.
The latter subgraph $s \not \ni e$ yields $\theta_{G \backslash e}$.
\end{proof}

\subsubsection{Relation to Tutte's V-function}
In 1947 \cite{Tring}, Tutte defined
a class of graph invariants called the V-function.
The definition is as follows.
\begin{de}
Let $\mathcal{G}$ be the set of isomorphism classes of finite undirected
 graphs, with loops and multiple edges allowed.
Let $R$ be a commutative ring.
A map $\mathcal{V} : \mathcal{G} \rightarrow R$
is called a {\it V-function}
if it satisfies the following two conditions:
\begin{align*}
&{\rm(i)} \quad 
\mathcal{V}(G)=\mathcal{V}(G \backslash e)+\mathcal{V}(G/e)
\qquad \text{if $e \in E$ is not a loop,} \qquad \qquad \\
&{\rm(ii)} \quad
\mathcal{V}(G_1 \cup G_2)= \mathcal{V}(G_1) \mathcal{V}(G_2). \qquad \qquad
\end{align*}
\end{de}

Our graph invariant $\theta$ is essentially an example of 
a V-function.
In the definition of V-functions, the coefficients of the
deletion-contraction relation are $1$, whereas those of $\theta$
are $(1-\beta)$ and $\beta$.
However, if we modify $\theta$ to
\begin{equation}
\hat{\theta}_{G}(\beta,\gamma):= (1-\beta)^{-|E|+|V|}\beta^{-|V|}
 \theta_{G}(\beta,\gamma) ,  \nonumber
\end{equation}
we obtain a V-function
$\hat{\theta} : \mathcal{G} \rightarrow  \mathbb{Z}[\beta,\gamma,\beta^{-1},(1-\beta)^{-1}]$.

In Theorem 10 of \cite{BPRcontraction}, Bollob\' as et al. 
have constructed non-isomorphic $k$-connected graphs that are not distinguished by deletion-contraction invariants. 
The result implies that these graphs have the same $\theta$-polynomial.

\subsubsection{Alternative expression of $\theta$-polynomial}
By successive applications of the conditions of V-function,
we can reduce the value at any graph to the values at bouquet graphs.
Therefore, we can say that a V-function is completely determined by its
boundary condition, i.e., the values at the bouquet graphs.
Conversely,
Tutte showed in \cite{Tring} that for an arbitrary boundary condition,
there exists a V-function that satisfies it.
More explicitly, the V-function satisfying a boundary condition
$\{ \mathcal{V}(B_n) \}_{n=0}$
is given by
\begin{equation}
 \mathcal{V}(G)= \sum_{s \subset E}
\prod_{n=0} z_n^{i_n(s)}, \label{V-funcTutterep}
\end{equation}
where $z_n:=\sum_{j=0}^{n}{n \choose j}(-1)^{n+j}\mathcal{V}(B_j)$ and
$i_n(s)$ is the number of connected components of the subgraph
$s$ with nullity $n$.

Note that another expansion called the spanning forest expansion of $\mathcal{V}(G)$
is described in Section 5 of \cite{BPRcontraction}.

In the case of $\theta$, Eq.~(\ref{V-funcTutterep})
derives the following expression.
Although this theorem is a trivial consequence of Theorem \ref{cortheta}
proved more directly later,
we present a proof of Theorem \ref{thmaltrep} to clarify the relation to Eq.~(\ref{V-funcTutterep}).
\begin{thm}
\label{thmaltrep}
\begin{equation}
\theta_{G}(\beta,\gamma)=
 \sum_{s \subset E}
\prod_{n=0}
\theta_{B_n}(1,\gamma)^{i_n(s)}
\beta^{|s|}
(1-\beta)^{|E|-|s|}. \label{thmaltrepeq}
\end{equation}
\end{thm}
\begin{proof}
It is sufficient to verify that
\begin{equation}
\hat{\theta}_{G}(\beta,\gamma)=
 \sum_{s \subset E}
\prod_{n=0}
{\theta}_{B_n}(1,\gamma)^{i_n(s)}
\beta^{|s|-|V|}
(1-\beta)^{|V|-|s|}. \label{thmaltrepeq2}
\end{equation}
By comparing the coefficients of $x^{k}$ in
$(1- \frac{1-x}{1-\beta})^{n}=(1-\beta)^{-n}(-\beta+x)^{n}$,
we have
\begin{equation}
 \sum_{j=k}^{n}(-1)^{j+n}
{n \choose j}{j \choose k}(1-\beta)^{-j}
=
{n \choose k} \beta^{n-k}(1-\beta)^{-n} \label{thmaltrepeq3}
\end{equation}
for every $0 \leq k \leq n$.
Using this equality and Proposition 1.(b), we see that
\begin{equation}
 z_n=
\sum_{j=0}^{n}{n \choose j}(-1)^{n+j} \hat{\theta}_{B_j}(\beta,\gamma)
=
\theta_{B_n}(1,\gamma)\beta^{n-1} (1-\beta)^{1-n}. \nonumber
\end{equation}
Therefore, Eq.~(\ref{V-funcTutterep}) reduces to Eq.~(\ref{thmaltrepeq2}).
\end{proof}

Formulae (\ref{defthetaeq}) and (\ref{thmaltrepeq}) are both 
represented in the sum of the subsets of edges;
however, the terms of a subset are different.
Generally, a V-function does not have a representation
corresponding to Eq.~(\ref{defthetaeq});
this representation
is utilized in the remainderr of the paper
and makes the $\theta$-polynomial worthy of investigation among all V-functions.

\subsubsection{Comparison with Tutte polynomial}
The most famous example of a V-function is the Tutte polynomial
(multiplied with a trivial factor).
The {\it Tutte polynomial} is defined by
\begin{equation}
 T_{G}(x,y):=\sum_{s \subset E}
(x-1)^{r(G)-r(s)}
(y-1)^{n(s)}.\label{defTutte}
\end{equation}
It satisfies a deletion-contraction relation
\begin{equation*}
T_{G}(x,y)=
\begin{cases}
 x T_{G \backslash e}(x,y)  \qquad \qquad \qquad \qquad \text{ if $e$ is a bridge,} \\
 y T_{G \backslash e}(x,y)  \qquad \qquad \qquad \qquad \text{ if $e$ is a loop,}   \\
 T_{G \backslash e}(x,y)+ T_{G / e}(x,y) \qquad \quad \text{ otherwise. }
\end{cases} 
\end{equation*}
It is easy to see that
$\hat{T}_{G}(x,y):=(x-1)^{k(G)}T_{G}(x,y)$ is a V-function
to $\mathbb{Z}[x,y]$.
For bouquet graphs, $\hat{T}_{B_n}(x,y)=(x-1)y^{n}$.
In the case of the Tutte polynomial, 
Eq.~(\ref{V-funcTutterep}) derives 
Eq.~(\ref{defTutte}).

Moreover, the Tutte polynomial $T$ is known to be 
matroidal, i.e., if $G_1$ and $G_2$ give the same cycle matroid,
then ${T}_{G_1}={T}_{G_2}$ holds \cite{Welsh}.
Because $B_{n+m}$ and $B_n \cup B_m$ give the same cycle matroid,
the relation 
\begin{equation}
 {T}_{B_{n+m}}= {T}_{B_n}{T}_{B_m} \label{eq:bouquetmulti}
\end{equation}
is a consequence of the invariance.
Strictly speaking, $\hat{T}$ is not matroidal; however, it satisfies Eq.~(\ref{eq:bouquetmulti}) 
up to the easy factor and is almost matroidal.

The V-functions 
$\hat{\theta}$ and $\hat{T}$ are 
essentially different.
One intuitive understanding is that 
$\hat{\theta}_{B_n}$, shown in Proposition \ref{propbasictheta}.(b),
do not satisfy Eq.~(\ref{eq:bouquetmulti}) 
even if multiplied with an appropriate factor.
(If we set $\gamma=0$, this is not the case. See Proposition \ref{thmgamma0}.)
In the following remark, we formally state the 
difference irrespective of transforms between $(\beta,\gamma)$ and $(x,y)$.

\begin{remark}
\label{propthetaTutte}
For any field $K$, inclusions
$\phi_1: \mathbb{Z}[\beta,\gamma,\beta^{-1},(1-\beta)^{-1}] \hookrightarrow K$,
and $\phi_2: \mathbb{Z}[x,y] \hookrightarrow K$,
we have
\begin{equation}
\phi_1 \circ \hat{\theta} \neq \phi_2 \circ \hat{T}. \nonumber
\end{equation}
\end{remark}
\begin{proof}
It is easy to see that
$\phi_2(\hat{T}_{B_n}) /\phi_2(\hat{T}_{B_0})= \phi_2(y)^{n}$ and
$\phi_1(\hat{\theta}_{B_n}) / \phi_1(\hat{\theta}_{B_0}) 
= \phi_1(1-\beta )^{-n} \phi_1(\sum_{k=0}^{n} {n \choose k} f_{2k}(\gamma) \beta^{k}).$
If $\phi_1 \circ \hat{\theta} = \phi_2 \circ \hat{T}$,
then
$a_n:=$
$\sum_{k=0}^{n} {n \choose k} f_{2k}(\gamma') \beta'^{k}$
$= z^{n}$
for some $z \in K$, where $\gamma'=\phi_1(\gamma)$ and $
 \beta'=\phi_1(\beta)$.
The equation $a_1^{2} = a_2$ gives $\gamma'^{2}\beta'^{2}=0$.
This is a contradiction because $\beta \neq 0$ and $\gamma \neq 0$.
\end{proof}

\section{Motivation for definition} \label{secMotivation}

In this section, we explain the motivation for considering the graph
polynomial $\theta_{G}$, that is,
the relation to the Ising partition function and its Bethe approximation.

\subsection{Definition of weighted graph version of $\theta$-polynomial}
We consider the multi-variable version of $\theta_{G}$
by attaching weights to the vertices and edges 
of $G$ by 
$\bs{\gamma}=(\gamma_{i})_{i \in V}$ and
$\bs{\beta}=(\beta_{e})_{e \in E}$ respectively.
Such a graph is called a {\it weighted graph}.
We assume that the weights are real numbers.

\begin{de}
\label{defTheta}
Let $\bs{\beta}=(\beta_{e})_{e \in E}$ 
 and $\bs{\gamma}=(\gamma_{i})_{i \in V}$
be the weights of $G$. 
\begin{equation}
\Theta_{G}(\bs{\beta},\bs{\gamma}):= 
\sum_{s \subset E}
\prod_{e \in s} 
\beta_{e}
\prod_{i \in V} f_{d_i(s)}(\gamma_i).   \nonumber
\end{equation}
\end{de}
If all vertex and edge weights are set to be the same,
$\Theta_{G}(\bs{\beta},\bs{\gamma})$ reduces to $\theta_{G}(\beta,\gamma)$.
It is trivial by definition that
\begin{align}
&\Theta_{G_1 \cup G_2}(\bs{\beta},\bs{\gamma})
 =\Theta_{G_1}(\bs{\beta},\bs{\gamma})\Theta_{G_2}(\bs{\beta},\bs{\gamma}), \label{Thetapropertymulti}\\
&\Theta_{B_0}(\bs{\beta},\bs{\gamma})=1, \label{Thetaproperty1} \\ 
&\Theta_{G}(\bs{\beta},\bs{\gamma})=\Theta_{\cg{G}}(\bs{\beta},\bs{\gamma}). \label{Thetaproperty2}
\end{align}

In this definition,
$\Theta_G$ is represented in the form of the edge states sum;
however, it is also possible to represent it in the following 
form of a vertex state sum.
This formula is important to show the relation to the Bethe
approximation of the Ising partition function 
because the partition function is also given in the form of a vertex state sum.
\begin{lem}
\label{thetaidentity}
\begin{equation}
\Theta_{G}(\bs{\beta},(\xi_{i}-\xi_{i}^{-1})_{i \in V})= 
\sum_{x_1,\ldots,x_N=\pm 1}
\prod_{e \in E \atop e=ij}
(1+x_{i}x_{j}\beta_{e}\xi_{i}^{-x_i}\xi_{j}^{-x_j})
\prod_{i \in V}
\frac{\xi_{i}^{x_i}}{\xi_{i}+\xi_{i}^{-1}}. \label{thetaidentityeq}
\end{equation}
\end{lem}
\begin{proof}
From Eq.~(\ref{defindf}), we can easily verify by induction that
\begin{equation}
f_n(\xi-\xi^{-1})=
\frac{\xi^{n-1}-(-\xi)^{-n+1}}{\xi_{}+\xi_{}^{-1}}. \nonumber
\end{equation}
If we expand the
product with respect to $E$ in the right-hand side of Eq.~(\ref{thetaidentityeq}), 
it is equal to
\begin{equation}
\sum_{s \subset E}
\prod_{e \in s}  
\beta_{e}
\prod_{i \in V} 
\sum_{x_i= \pm 1}
\frac{(-x_i)^{d_{i(s)}} \xi_{i}^{(1-d_{i(s)})x_i}}{\xi_{i}+\xi_{i}^{-1}}. \nonumber
\end{equation}
Then,
the assertion follows immediately.
\end{proof}

\subsection{Relation to Bethe approximation}\label{sec:linkBethe}
We demonstrate that the value $\Theta_{G}$
describes the discrepancy between the true partition function
of the Ising model
and its Bethe approximation.
A more detailed discussion of the same is found in \cite{WFloop}. 

The Bethe approximation is a method used for approximating partition
functions of various statistical mechanical models \cite{Bethe}.
Here, we state it for the case of the Ising partition function.
Recall that 
the {\it Ising partition function} on $G$
for given $\bs{J}=(J_e)_{e \in E}$ and $\bs{h}=(h_i)_{i \in V}$
is defined by Eq.~(\ref{defIsing}).
We write $\psi_{ij}(x_i,x_j)=\exp(J_{ij} x_i x_j)$ and
$\psi_{i}(x_i)=\exp(h_i x_i)$.

\begin{de}
\label{deBethe}
A set of functions $\{b_{e}(x_i,x_j)\}_{e \in E}$ and
 $\{b_{i}(x_i)\}_{i \in V}$ is called a 
{\it belief} \cite{YFWGBP}
if it satisfies
\begin{align}
&\sum_{x_i}b_{e}(x_i,x_j) = b_{i}(x_i)  
\quad \text{ for all } i \in V, \ x_i \in \{\pm 1\}, \text{ and } e=ij \in E, \label{deBetheeq1} \\
&\sum_{x_i,x_j} b_{e}(x_i,x_j) =1 
\quad \quad \text{ for all } e=ij \in E, \label{deBetheeq2} \\
&\prod_{e \in E}
\frac{b_{e}(x_i,x_j)}{b_{i}(x_i)b_{j}(x_j)}
\prod_{i \in V}b_{i}(x_i)^{}
\propto
\prod_{e \in E}\psi_{e}(x_i,x_j)
\prod_{i \in V}\psi_{i}(x_i). \label{deBetheeq3}
\end{align}
Then, the {\it Bethe approximation of the partition function}
$Z_B$ is defined by the proportionality constant of 
Eq.~(\ref{deBetheeq3}):
$Z_B\prod_{e \in E}
\frac{b_{e}}{b_i b_j}
\prod_{i \in V}b_{i}^{}
=
\prod_{e \in E}\psi_{e}
\prod_{i \in V}\psi_{i}$.
\end{de}
For given $\bs{J}$ and $\bs{h}$,
we can obtain a belief
by an algorithm called {\it belief propagation} \cite{Pearl,YFWGBP}.
In practical situations, the algorithm stops in a reasonable time.
Therefore, the Bethe approximation of the partition function is 
used in many applications \cite{Turbo}.

We show that 
$\Theta_{G}(\bs{\beta},\bs{\gamma})$
is equal to $Z/Z_B$. 
We choose variables $\beta_e$ and $\xi_i$
to parameterize $\{b_{e}(x_i,x_j)\}_{e \in E}$ and
 $\{b_{i}(x_i)\}_{i \in V}$,
which satisfy Eqs.~(\ref{deBetheeq1}) and (\ref{deBetheeq2}):
\begin{align*}
&b_{e}(x_i,x_j)=\frac{1}{(\xi_i+\xi_i^{-1})(\xi_j+\xi_j^{-1})}
(\xi_i^{x_i}\xi_j^{x_j}+\beta_e x_i x_j),\\
&b_{i}(x_i)=\frac{\xi_i^{x_i}}{\xi_i+\xi_i^{-1}}. 
\end{align*}
From the definition of $Z_B$ and Lemma \ref{thetaidentity}, we see that 
\begin{align*}
\frac{Z}{Z_B}
&=
\sum_{\bs{x}}
\prod_{e \in E}
\frac{b_{e}(x_i,x_j)}{b_{i}(x_i)b_{j}(x_j)}
\prod_{i \in V}b_{i}(x_i)^{}   \\
&=
\sum_{x_1,\ldots,x_N=\pm 1}
\prod_{e \in E \atop e=ij}
(1+x_{i}x_{j}\beta_{e}\xi_{i}^{-x_i}\xi_{j}^{-x_j})
\prod_{i \in V}
\frac{\xi_{i}^{x_i}}{\xi_{i}+\xi_{i}^{-1}} \\
&=\Theta_{G}(\bs{\beta},\bs{\gamma}),
\end{align*}
where $\gamma_i:=\xi_i - \xi_i^{-1}$.
This equation implies that the approximation ratio is captured by 
the value of $\Theta_{G}$.
If the graph is a tree,
we see from Eq.~(\ref{Thetaproperty1}) and (\ref{Thetaproperty2})
that $\Theta_{G}=1$,
i.e., the Bethe approximation gives the exact value of the partition
function.
If the weights $\bs{\beta}$ and $\bs{\gamma}$ are sufficiently small,
we see that $\Theta_{G} \approx 1$, i.e., the Bethe approximation is a
good approximation.

The definition of $\Theta_{G}$ implies that
we can expand the approximation ratio by the sum of sub-coregraphs 
\cite{CCloopPRE,CCloop,WFloop}.
This expansion often improves the approximation if we sum up some of
the terms \cite{GMKtruncate}.

\subsection{Transform of Ising partition function} \label{sec:transformIsing}
In the following, we give the explicit transform from
$(\bs{\beta},\bs{\gamma})$ to $(\bs{J},\bs{h})$.

We can always choose $A_i,B_e,h^{'}_i,h_{e,i}$, and $J_e$ to satisfy
\begin{align*}
&\frac{\xi_{i}^{x_i}}{\xi_{i}+\xi_{i}^{-1}} 
=
A_i^{-1} \exp (h_i^{'} x_i), \\
&1+x_{i}x_{j}\beta_{e} \xi_{i}^{-x_i}\xi_{j}^{-x_j}
=
B_{e}^{-1} 
\exp
(J_e x_i x_j + h_{e,i}x_i+h_{e,j}x_j).
\end{align*}  
Therefore, setting 
$h_i:=h_i^{'}+\sum_{e \backepsilon i}h_{e,i}$,
we have
\begin{equation}
Z(G;\bs{J},\bs{h}) 
=
\prod_{i \in V}A_i
\prod_{e \in E}B_e
\
\Theta_{G}(\bs{\beta},(\xi_{i}-\xi_{i}^{-1})_{i \in V}). 
\quad
\label{thetaisingeq}
\end{equation} 
This fact shows that $\Theta_G(\bs{\beta},\bs{\gamma})$ gives the Ising
partition function with $(\bs{J},\bs{h})$, which is computed 
from $(\bs{\beta},\bs{\gamma})$ as above.

If $\xi_i=1$, or $\gamma_i=0$  for all $i \in V$,
Eq.~(\ref{thetaisingeq}) reduces to the well-known expansion given by
van der Waerden \cite{Waerden,Welsh},
\begin{equation}
Z(G;\bs{J},0)  
=
2^{|V|}
\prod_{e \in E}\cosh (J_e)
\sum_{s \in \mathcal{E}}
\prod_{e \in s} \tanh (J_e), \label{VDW}
\end{equation} 
where $\mathcal{E}$ is the set of Eulerian subgraphs,
i.e., subgraphs in which all vertex degrees are even.
This fact is deduced from $f_n(0)=1$ if $n$ is even and 
$f_n(0)=0$ if $n$ is odd.

It is well known by statistical physicists
that Eq.~(\ref{VDW}) can be extended to
the following expression \cite{Dcooperative}:
\begin{equation}
 Z(G;\bs{J}, \bs{h}) = 
2^{|V|}
\prod_{e \in E}\cosh (J_e)
\sum_{s \subset E}
\prod_{e \in s} \tanh (J_e)
\prod_{i \in V_{e}(s)} \hspace{-2mm} \cosh (h_i)
\prod_{i \in V_{o}(s)} \hspace{-2mm} \sinh (h_i), \label{tradising}
\end{equation}
where $V_{e}(s)$ (resp. $V_{o}(s)$) is the set of vertices of even (resp. odd) degree
in $s$.
Although both Eqs.~(\ref{thetaisingeq}) and (\ref{tradising})
are extensions of Eq.~(\ref{VDW}) and give edge subset expansions,
they are different. 
An obvious difference is that only the sub-coregraphs contribute to the expansion in Eq.~(\ref{thetaisingeq}).

Based on  Eq.~(\ref{thetaisingeq}),
we can say that the graph polynomial $\theta_G(\beta,\gamma)$ is a transformed Ising partition function
with uniform coupling constants and non-uniform external fields.
In contrast, 
a bivariate graph polynomial investigated in \cite{AMbivariate}
is based on Eq.~(\ref{tradising}).
This polynomial corresponds to the Ising partition function with uniform coupling constants and external fields.
A similar type of expression is also considered in \cite{Lchromatic}.

\subsection{Additional remarks on weighted graph version}
In this subsection, we present additional remarks on $\Theta_G$
by comparing it with $\theta_{G}$. 
The deletion-contraction relation given by Theorem \ref{thmcd}
is generalized to weighted graphs as follows.
If the weights $(\bs{\beta},\bs{\gamma})$ on $G$ satisfy 
$\gamma_{i}=\gamma_{j}$ for a non-loop edge $e=ij$,
the weights on $G \backslash e$ and $G/e$
are naturally induced and denoted by $(\bs{\beta'},\bs{\gamma'})$
and $(\bs{\beta''},\bs{\gamma''})$, respectively. 
On $G/e$, 
the weight on the new vertex, which is the fusion of $i$ and
$j$, is set to be $\gamma_{i}$.
Under these conditions, we have
\begin{equation}
 \Theta_{G}(\bs{\beta},\bs{\gamma})=
(1-\beta_e) \Theta_{G\backslash e}(\bs{\beta'},\bs{\gamma'}) +
\beta_{e}  \Theta_{G/e}(\bs{\beta''},\bs{\gamma''}),  \label{weightedcd}
\end{equation}
which is proved in the same manner as Theorem \ref{thmcd}.

If we set all vertex weights 
$\gamma_i$ to be equal,
the generalization of Theorem \ref{thmaltrep} holds.
We write 
$\Theta_{G}(\bs{\beta},(\gamma_i = \gamma)_{i \in V})$ by
$\Theta_{G}(\bs{\beta},\gamma)$ 
for simplicity.

\begin{thm} 
\label{cortheta}
\begin{equation}
 \Theta_{G}(\bs{\beta},\gamma)=
\sum_{s \subset E}
\prod_{n=0}
\theta_{B_n}(1,\gamma)^{i_n(s)}
\prod_{e \in s}
\beta_{e}
\prod_{e \in  E \setminus s}
(1-\beta_e ).  \label{corthetaeq1}
\end{equation}
\end{thm}
\begin{proof}
In this proof,
the right-hand side of Eq.~(\ref{corthetaeq1}) is 
denoted by $ \tilde{\Theta}_{G}(\bs{\beta},\gamma)$.
First, we check that ${\Theta}_{G}$ and 
$\tilde{\Theta}_{G}$ are equal at the bouquet graphs.
\begin{align}
\tilde{\Theta}_{B_n}(\bs{\beta},\gamma)
&=
\sum_{s \subset E}
\theta_{B_{|s|}}(1,\gamma)
\prod_{e \in s}
\beta_{e}
\prod_{e \in  E \setminus s}
(1-\beta_e ) \nonumber \\
&=
\sum_{s \subset E}
\sum_{k =0}^{|s|}
{|s| \choose k}f_{2k}(\gamma)
\prod_{e \in s}
\beta_{e}
\sum_{t \subset E \setminus s}
\prod_{e \in t}
(-\beta_e ) \nonumber \\
&=
\sum_{u \subset E}
\sum_{s \subset u}
\sum_{k =0}^{|s|}
{|s| \choose k}f_{2k}(\gamma)
(-1)^{|u|-|s|}
\prod_{e \in u}
\beta_e  \nonumber \\
&=
\sum_{u \subset E}
\sum_{l=0}^{|u|}
\sum_{k =0}^{l}
{|u| \choose l}
{l \choose k}f_{2k}(\gamma)
(-1)^{|u|-l}
\prod_{e \in u}
\beta_e. \nonumber
\end{align}
Using the equality
$\sum_{j=k}^{n} {n \choose j}{j \choose k} (-1)^{n+j}=\delta_{n,k}$,
which is obtained at $\beta=0$
of Eq.~(\ref{thmaltrepeq3}),
we have
\begin{equation}
\tilde{\Theta}_{B_n}(\bs{\beta},\gamma)
=
\sum_{u \subset E} f_{2|u|}(\gamma) 
\prod_{e \in u}\beta_{e} 
=
\Theta_{B_n}(\bs{\beta},\gamma). \nonumber
\end{equation}

Second, we see that $ \tilde{\Theta}_{G}(\bs{\beta},\gamma)$
satisfies the deletion-contraction relation 
\begin{equation}
 \tilde{\Theta}_{G}(\bs{\beta},\gamma)=
(1-\beta_e) \tilde{\Theta}_{G\backslash e}(\bs{\beta'},\gamma) +
\beta_{e}  \tilde{\Theta}_{G/e}(\bs{\beta''},\gamma) \nonumber
\end{equation}
for all non-loop edges $e$,
because
the subsets including $e$ amount to
$\beta_e \tilde{\Theta}_{G / e}(\bs{\beta},\gamma)$ 
and
the other subsets
amount to
$(1-\beta_e) \tilde{\Theta}_{G \setminus e}(\bs{\beta},\gamma)$.

By applying this form of deletion-contraction relations to both $\Theta_{G}$
and $\tilde{\Theta}_{G}$, we can reduce the values at $G$ to those of  
disjoint unions of bouquet graphs.
Therefore, we conclude that $\tilde{\Theta}_{G}=\Theta_{G }$.
\end{proof}

A {\it coloured} graph is a graph with a map from the edges to a set of colours.
If it is the set of real numbers, the term ``weighted'' is preferred.  
We can generalize the definition of V-functions to coloured graphs by allowing the coefficients of the deletion-contraction relation
to depend on colours.
Since ${\Theta}_{G}(\bs{\beta},\gamma)$ satisfies Eqs.~(\ref{Thetapropertymulti}) and (\ref{weightedcd}),  
it is a V-function of (edge) weighted graphs.
An expansion similar to Theorem \ref{cortheta} holds for any coloured V-function 
because the proof only uses Eqs.~(\ref{Thetapropertymulti}) and (\ref{weightedcd}).

With regard to the Tutte polynomial, numerous works have focused on its extensions to edge-weighted or coloured versions.
In \cite{BRTutte}, the ``universal'' Tutte polynomial is constructed on coloured graphs
by generalizing the ordinary Tutte polynomial to the greatest extent possible.
The ``universal'' Tutte polynomial derives other extensions of the Tutte polynomial such as 
the dichromatic polynomial for edge-weighted graphs given by Traldi \cite{Tdichromatic} and 
the {\it random-cluster model} by given by Fortuin and Kasteleyn \cite{FKrandom}.

Our extension, ${\Theta}_{G}(\bs{\beta},\gamma)$, for weighted graphs
resembles the random-cluster model defined by
\begin{equation}
R_{G}(\bs{\beta},\kappa)
=
\sum_{s \subset E}
\kappa^{k(s)}
\prod_{e \in s}
\beta_{e}
\prod_{e \in  E \setminus s}
(1-\beta_e ) \nonumber
\end{equation}
because of Eq.~(\ref{corthetaeq1}).
The random-cluster model satisfies a deletion-contraction relation of the form
\begin{equation}
 R_{G}(\bs{\beta},\kappa) = (1-\beta_e) R_{G \backslash e}(\bs{\beta'},\kappa) 
+\beta_e R_{G/e}(\bs{\beta''},\kappa) 
\quad \text{ for all } e \in E. \nonumber
\end{equation}
Note that this relation holds for loops in contrast to
$\Theta_{G}(\bs{\beta},\gamma)$ as
$R_{G}(\bs{\beta},\kappa)$ is an extension of the Tutte polynomial.
This difference arises from that of the coefficients of subgraphs $s$:
$\kappa^{k(s)}$ and $\prod \theta_{B_n}(1,\gamma)^{i_n(s)}$.

\section{Additional properties of $\theta$ and its implications} \label{secfurther}

\subsection{Special values}
\subsubsection{$\gamma =0$ case}
As suggested 
in Section \ref{propthetaTutte},
if we set $\gamma=0$, the polynomial $\theta_{G}(\beta,0)$ is included in 
the Tutte polynomial.
\begin{prop}
\label{thmgamma0}
\begin{equation*}
 \theta_{G}(\beta,0)=
(1-\beta)^{n(G)}\beta^{r(G)}
T_{G}
\Big(
\frac{1}{\beta},\frac{1+\beta}{1-\beta}
\Big).
\end{equation*}
\end{prop}
\begin{proof}
From Proposition \ref{propbasictheta}.(b)
and $f_{2k}(0)=1$, we have
\begin{equation*}
 \hat{\theta}_{B_n}(\beta,0)=(1-\beta)^{1-n}\beta^{-1}
\sum_{k=0}^{n} {n \choose k} \beta^{k} =(1-\beta)^{1-n}\beta^{-1} (1+\beta)^{n}.
\end{equation*}
We also have 
$\hat{T}_{B_n}(\frac{1}{\beta},\frac{1+\beta}{1-\beta})=(\beta^{-1}-1)(\frac{1+\beta}{1-\beta})^{n}$.
Therefore, 
$\hat{\theta}_{B_n}(\beta,0)=\hat{T}_{B_n}(\frac{1}{\beta},\frac{1+\beta}{1-\beta})$.
Because V-functions are determined by the values at the bouquet graphs,
 $\hat{\theta}_{G}(\beta,0)=\hat{T}_{G}(\frac{1}{\beta},\frac{1+\beta}{1-\beta})$
 holds for any graph $G$.
\end{proof}

This result is natural from the viewpoint of the Ising partition function.
The Tutte polynomial is equivalent to the partition
function of the q-Potts model \cite{Bollobas};
if we set $q=2$, it becomes the Ising partition function (with uniform
coupling constants $J$ and without external fields).
In terms of the Tutte polynomial,
such points correspond to the parameters
$(x,y)=(\frac{1}{\beta},\frac{1+\beta}{1-\beta})$,
and thus, $T_{G}(\frac{1}{\beta},\frac{1+\beta}{1-\beta})$
is essentially the Ising partition function of that type.
On the other hand, as discussed in Section \ref{sec:transformIsing},
$\theta_{G}(\beta,0)$ is also essentially equal to the Ising partition function of that type.
Therefore they must be equal up to some easy factor.

We can say that the Tutte polynomial is an extension of 
the Ising partition function 
(with uniform coupling constants and without external fields)
to the $q$-state model 
whereas
the  $\theta$-polynomial is an extension of it
to a model with specific forms of local external fields.

\subsubsection{$\beta = 1$ case}
At $\beta=1$, $\theta_{G}(1,\gamma)$ is determined by the nullity and the number of connected components
of the graph.
\begin{lem} 
\label{lembeta1}
For a connected graph $G$,
\begin{equation}
\theta_{G}(1,\xi-\xi^{-1})
=
\xi^{1-n(G)}(\xi+\xi^{-1})^{n(G)-1}
+
\xi^{n(G)-1}(\xi+\xi^{-1})^{n(G)-1}.
\label{lembeta1eq1}
\end{equation}
\end{lem}
\begin{proof}
We use
the right-hand side of Lemma \ref{thetaidentity}, which gives an alternative representation
of $\theta_{G}$.
If $x_i \neq x_j$, then $1+x_{i}x_{j}\xi^{-x_i}\xi^{-x_j}=0$.
Thus, only two terms of
 $x_1=\cdots=x_N=1$ and $x_1=\cdots=x_N=-1$ contribute to the sum,
because $G$ is connected.
\end{proof}

If $\xi=\frac{1+\sqrt{5}}{2}$, then $\xi-\xi^{-1}=1$. From Eq.~(\ref{lembeta1eq1}),
we see that
\begin{equation}
\theta_{G}(1,1)=
 \Biggl(\frac{5-\sqrt{5}}{2}\Biggr)^{n(G)-1}
\hspace{-2mm}
+
\Biggl(\frac{5+\sqrt{5}}{2}\Biggr)^{n(G)-1}. \label{numup}
\end{equation}
Setting $\xi=1$,
we also deduce from Eq.~(\ref{lembeta1eq1}) that
\begin{equation}
\theta_{G}(1,0)=2^{n(G)}. \label{numlow}
\end{equation}

\subsection{Number of sub-coregraphs}
\subsubsection{Bounds}
For a given graph $G$,
let
$\mathcal{C}(G):=\{ s; s \subset  E, (V,s) \text{ is a coregraph.} \}$
be the set of sub-coregraphs of $G$. 
In the following theorem,
the values (\ref{numup}) and (\ref{numlow})
are used to bound the number of sub-coregraphs.

Although the following upper bound is proved in \cite{WFloop},
here, we present both the proofs of the bounds for completeness.

\begin{thm}
\label{thmbound}
For a connected graph $G$,
\begin{equation}
2^{n(G)}
\leq
\left| \mathcal{C}(G) \right|
\leq
 \Biggl(\frac{5-\sqrt{5}}{2}\Biggr)^{n(G)-1}
\hspace{-2mm}
+
\Biggl(\frac{5+\sqrt{5}}{2}\Biggr)^{n(G)-1}. \label{thmboundeq1}
\end{equation}
The lower bound is attained if and only if
$\cg{G}$ is a subdivision of a bouquet graph,
and the upper bound is attained if and only if 
$\cg{G}$ is a subdivision of a 3-regular graph
or $G$ is a tree.
\end{thm}
Note that a {\it subdivision} of a graph $G$ is 
a graph that is obtained by adding vertices
of degree 2 on edges.
\begin{proof}
It is sufficient to consider the case in which $G$ is a coregraph
and does not have vertices of degree 2,
because 
the operations of taking core and subdivision
do not essentially change the nullity and the set of sub-coregraphs.

From the definition of Eq.~(\ref{defthetaeq}), we can write
\begin{equation}
\theta_{G}(1,\gamma)
=
\sum_{s \in \mathcal{C}} w(s;\gamma),  \nonumber
\end{equation}
where $w(s;\gamma)=\prod_{i \in V}f_{d_{i}(s)}(\gamma)$.
For all $s \in \mathcal{C}$,
we claim that
\begin{equation}
w(s;0) \leq 1 \leq w(s;1). \label{thmboundeq2}
\end{equation}
The left inequality of Eq.~(\ref{thmboundeq2})
is immediate from the fact that $f_n(0)=1$ if $n$ is even
and $f_n(0)=0$ if $n$ is odd.
The equality holds if and only if
all vertices have even degree in $s$. 
Because $f_n(1)>1$ for all $n>4$ and $f_2(1)=f_3(1)=1$,
we have $w(s;1) \geq 1$.
The equality holds if and only if 
$d_i(s) \leq 3$
for all $i\in V$.
Then,
the inequalities in Eq.~(\ref{thmboundeq1}) 
are proved.
The upper bound is attained
if and only if $G$ is a 3-regular graph or $B_0$.
For the equality condition of the lower bound,
it is sufficient to prove the following claim.

\begin{claim}
Let $G$ be a connected graph, and assume that the degree of every vertex is 
at least 3 and $d_i(s)$ is even for every $i \in V$ and $s \in \mathcal{C}$.
Then, $G$ is a bouquet graph. 
\end{claim}

If $G$ is not a bouquet graph, there exists a non-loop edge $e=i_0j_0$.
Then, $E$ and $E \setminus e$ are sub-coregraphs of $G$.
Thus, $d_{i_0}(E)$ or $d_{i_0}(E \setminus e)=d_{i_0}(E)-1$ is odd.
This is a contradiction.
\end{proof}

\subsubsection{Number of sub-coregraphs in 3-regular graphs}
If the core of a graph is a subdivision of a 3-regular graph,
we obtain more information on the number of specific types of sub-coregraphs. 

We can rewrite Lemma \ref{lembeta1} as follows.
\begin{lem} 
\label{lemthetacoeff}
Let $G$ be connected and not a tree. Then, we have
\begin{equation}
\theta_{G}(1,\gamma)
=
\sum_{l=0}^{n(G)-1} C_{n(G),l} \gamma^{2l},  \nonumber
\end{equation}
where $C_{n,l}:=\sum_{k=l+1}^{n}\binom{n}{k}\binom{k+l-1}{2l}$ for 
$1 \leq l \leq n-1$ and $C_{n,0}:= 2^{n}$.
\end{lem}
\begin{proof}
First, we note that
for $k \geq 1$,
\begin{equation}
f_{2k}(\gamma)=\sum_{l=0}^{k-1}\binom{k+l-1}{2l}\gamma^{2l}  
\quad  \text{  and  } \quad
f_{2k+1}(\gamma)=\sum_{l=0}^{k-1}\binom{k+l}{2l+1}\gamma^{2l+1}.  \nonumber
\end{equation}
This is easily proved inductively using
Eq.~(\ref{defindf}).
Then, Lemma \ref{lembeta1} gives
\begin{align*}
\theta_{G}(1,\gamma)
=
\theta_{B_{n(G)}}(1,\gamma)
&=\sum_{k=1}^{n(G)} \binom{n(G)}{k} f_{2k}(\gamma)
+f_{0}(\gamma) \\
&=
\sum_{l=0}^{n(G)-1}\sum_{k=l+1}^{n(G)} \binom{n(G)}{k} \binom{k+l-1}{2l}\gamma^{2l}
+1 \\
&=
\sum_{l=0}^{n(G)-1} C_{n(G),l} \gamma^{2l}.
\end{align*}
\end{proof}

\begin{thm}
Let $G$ be a connected graph and not a tree.
If every vertex of the $\cg{G}$ has a degree of at most $3$,
then 
\begin{equation}
C_{n(G),l}=  
|\{ s \in \mathcal{C}(G) ; s \text{ has }2l \text{ vertices of degree
3.} \}|  \nonumber
\end{equation}
for $0 \leq l \leq n(G)-1$.
\end{thm}
\begin{proof}
For a sub-coregraph $s$,
$\prod_{i \in V}f_{d_{i}(s)}(\gamma)=\gamma^{2l}$ 
if and only if $s$ has $2l$ vertices of degree $3$.
\end{proof}

\section{One-variable graph polynomial $\omega$} \label{secomega}

In this section, we define the second graph polynomial
$\omega$ by setting $\gamma=2\sqrt{-1}$.
Using Eq.~(\ref{defindf}),
it is easy to verify that
$f_n(2\sqrt{-1})=(\sqrt{-1})^{n}(1-n)$.
Therefore,
\begin{equation}
\theta_{G}(\beta,2\sqrt{-1})
=
\sum_{s \subset E} 
(- \beta)^{|s|}
\prod_{i \in V}(1-d_i(s)). \label{theta2i} \\ 
\end{equation}
An interesting point of this specialization is 
the relation to the monomer-dimer partition function
with the specific form of monomer-dimer weights,
as described in Section \ref{subsectionomegamonomerdimer}.

\subsection{Definition and basic properties}
From Eq.~(\ref{lembeta1eq1}), $\theta_G(1,2\sqrt{-1})=0$
unless all the nullities of connected components of $G$ are less than $2$.
The following theorem asserts that $\theta_G(\beta,2\sqrt{-1})$ can be divided by
$(1-\beta)^{|E|-|V|}$.
We define $\omega_G$ by dividing that factor.
\begin{thm}
\label{thmomega}
\begin{equation}
 \omega_{G}(\beta):= \frac{\theta_{G}(\beta,2 \sqrt{-1})}{(1-\beta)^{|E|-|V|}} \quad
 \in \mathbb{Z}[\beta].  \nonumber
\end{equation}
\end{thm}
In Eq.~(\ref{theta2i}), $\theta_{G}(\beta,2 \sqrt{-1})$ is given in the summation over all
sub-coregraphs and each term is not necessarily divisible by $(1-\beta)^{|E|-|V|}$;
however, if we use the representation in Theorem \ref{thmaltrep}, each summand
is divisible by the factor, 
as shown in the following theorem.
Theorem \ref{thmomega} is a trivial consequence of Theorem \ref{thmomegaaltrep}.
\begin{thm}
\label{thmomegaaltrep}
\begin{equation}
\omega_{G}(\beta)
=
\sum_{s \subset E } 
\beta^{|s|}
\prod_{n=0}
h_{n}(\beta)^{i_n(s)}, \nonumber
\end{equation}
where $h_0(\beta):=(1-\beta)$, $h_1(\beta):=2$, and 
$h_n(\beta):=0$ for $n \geq 2$. 
\end{thm}
\begin{proof}
From (b) of Proposition \ref{propbasictheta} and
$f_m(2\sqrt{-1})=(\sqrt{-1})^{m}(1-m)$, we have
\begin{align}
\theta_{B_n}(1,2\sqrt{-1})
=
\sum_{k=0}^{n}{n \choose k}
(-1)^{k}(1-2k)
=
\begin{cases}
1 
\quad \text{  if }n=0 \\
2
\quad \text{  if }n=1 \\
0
\quad \text{  if }n \geq 2.
\end{cases}  \nonumber
\end{align}
Theorem \ref{thmaltrep} gives
\begin{align*}
\omega_{G}(\beta)
&=
\sum_{s \subset E}
\prod_{n=0}
\theta_{B_n}(1,2\sqrt{-1})^{i_n(s)}
\beta^{|s|}
(1-\beta)^{|V|-|s|} \\
&=
\sum_{s \subset E}
\prod_{n=0}
[(1-\beta)^{1-n}
\theta_{B_n}(1,2\sqrt{-1})
]^{i_n(s)}
\beta^{|s|}.
\end{align*}
Then, the assertion is proved.
\end{proof}

\begin{ex}
\label{exampleomega}
$\\$
For a tree $T$, $\omega_{T}(\beta)=1-\beta$.
For the cycle graph $C_n$,
$\omega_{C_n}(\beta)=1+\beta^{n}$.
For the complete graph $K_4$,
$\omega_{K_4}(\beta)=1+2\beta+3\beta^{2}+8\beta^{3}+16\beta^{4}$.
For the graphs shown in Figure \ref{X1X2},
$\omega_{X_1}(\beta)=1+\beta+4\beta^{2}$ and
$\omega_{X_2}(\beta)=1+3\beta+4\beta^{2}$.
\end{ex}

We list the basic properties of $\omega$ below.
\begin{prop}$\quad$ 
\label{propbasicomega}
\begin{itemize}
\item[{\rm (a)}] \ \
$\omega_{G_1 \cup G_2}(\beta)=\omega_{G_1}(\beta)\omega_{G_2}(\beta). $
\item[{\rm (b)}] \ \
$\omega_{G}(\beta)=\omega_{G \backslash e}(\beta)+\beta
\omega_{G / e}(\beta )
\quad
\text{if } e \in E \text{ is not a loop.}$
\item[{\rm (c)}] \ \
$\omega_{B_n}(\beta)=1+(2n-1)\beta$.
\item[{\rm (d)}] \ \
$ \omega_{G}(\beta)=\omega_{\cg{G}}(\beta). $
\item[{\rm (e)}] \ \
$\omega_{G}(\beta)$ is a polynomial of degree $|V_{\cg{G}}|$.
The leading coefficient is $\prod_{i \in V_{\cg{G}}} (d_i -1)$
and the constant term is $1$.
\item[{\rm (f)}] \ \
Let $G^{(m)}$ be the graph obtained by subdividing each edge to m edges. Then,
 \begin{equation*}
 \omega_{G^{(m)}}(\beta)=(1+\beta+ \cdots + \beta^{m-1})^{|E|-|V|}\omega_{G}(\beta^m).
 \end{equation*}
\end{itemize}
\end{prop}
\begin{proof}
Assertions (a$-$e) are easy.
(f) is proved by $|E_G|-|V_G|=|E_{G^{(m)}}|-|V_{G^{(m)}}|$ and
$\theta_{G^{(m)}}(\beta,2\sqrt{-1})
=\theta_{G}(\beta^{m},2\sqrt{-1})$.
\end{proof}

\begin{prop} 
\label{propomeganonneg}
If $G$ does not have connected components of nullity $0$, then
the coefficients of $\omega_{G}(\beta)$ are non-negative.
\end{prop}
\begin{proof}
We prove this assertion by induction on the number of edges.
Assume that every connected component is not a tree.
If $G$ has only one edge,
then $G=B_1$ and the coefficients are non-negative.
Let $G$ have $M (\geq 2)$ edges and assume that
the assertion holds for graphs with at most $M-1$ edges.
It is sufficient to consider the case in which $G$ is a connected coregraph
because of Proposition \ref{propbasicomega}.(a) and (d).
If all edges of $G$ are loops, $G=B_n$ for some $n \geq 2$ 
and the coefficients are non-negative.
If $G=C_M$, the coefficients are also non-negative, as in the case of Example
\ref{exampleomega}. 
Otherwise, we reduce $\omega_{G}$ 
to graphs with nullity not less than $1$
by applying the deletion-contraction relation 
and see that the coefficients of 
$\omega_{G \backslash e}$ and $\omega_{G/e}$ are both non-negative.
\end{proof}

\subsection{Relation to monomer-dimer partition
  function} \label{subsectionomegamonomerdimer}
In the next theorem, we prove that
the polynomial $\omega_{G}(\beta)$
is the monomer-dimer partition function with a specific form of weights.

A {\it matching} of $G$ is a set of edges such that no two edges occupy the same vertex.
It is also called a {\it dimer arrangement} in statistical
physics \cite{HLmonomerdimer}.
We use both terminologies.
The number of edges in a matching ${\bf D}$ is denoted by $|{\bf D}|$.
If a matching ${\bf D}$ consists of $k$ edges, then it is called a {\it k-matching}. 
The vertices covered by the edges in ${\bf D}$ are denoted by
$[{\bf D}]$. 
The set of all matchings of $G$ is denoted by $\mathcal{D}$.

The monomer-dimer partition function with edge weights 
$\bs{\mu}=(\mu_{e})_{e \in E}$ and vertex weights 
$\bs{\lambda}=(\lambda_{i})_{i \in V}$
is defined as
\begin{equation}
\Xi_G(\bs{\mu},\bs{\lambda})
:=
\sum_{{\bf D} \in \mathcal{D}}
\prod_{e \in {\bf D}} \mu_{e}
\prod_{i \in V \backslash [{\bf D}]}  \lambda_{i}.  \nonumber
\end{equation} 
We write $\Xi_G(\mu,\bs{\lambda})$ if all weights $\mu_{e}$ are set
to the same $\mu$.

\begin{thm} 
\label{thmmonomer}
Let $\lambda_{i}:=1+(d_i - 1) \beta$; then,
\begin{equation}
\omega_{G}(\beta)
=
\Xi_{G}(- \beta, \bs{\lambda}).  \nonumber
\end{equation}
\end{thm}
\begin{proof}
We show that $\Xi_{G}(- \beta, \bs{\lambda})$ satisfies 
the deletion-contraction relation and the boundary condition of the form in 
Proposition \ref{propbasicomega}.(c).
For the bouquet graph $B_n$, ${\bf D}= \phi$ is the only possible
dimer arrangement, and thus,
\begin{equation}
\Xi_{B_n}(- \beta, \bs{\lambda})=1+(2 n-1)\beta =\omega_{B_n}(\beta). \nonumber
\end{equation}
For a non-loop edge $e=i_0j_0$, we show that the deletion-contraction
 relation is satisfied.
A dimer arrangement ${\bf D} \in \mathcal{D}$ is classified into the following
 five types: 
(a) ${\bf D}$ includes $e$, 
(b) ${\bf D}$ does not include $e$ and ${\bf D}$ covers both $i_0$ and $j_0$, 
(c) ${\bf D}$ covers $i_0$ but not $j_0$, 
(d) ${\bf D}$ covers $j_0$ but not $i_0$,
and 
(e) ${\bf D}$ covers neither $i_0$ nor $j_0$.  
According to this classification, $\Xi_{G}(- \beta, \bs{\lambda})$ is a
 sum of the five terms $A,B,C,D$, and $E$.
We see that
\begin{align*}
C
=
\sum_{{\bf D} \in \mathcal{D} \atop [{\bf D}] \ni i_0, [{\bf D}] \not\ni j_0}
&(- \beta)^{|{\bf D}|}
\prod_{i \in V \backslash [{\bf D}]}
\lambda_{i}
\\
=
\sum_{{\bf D} \in \mathcal{D} \atop [{\bf D}] \ni i_0, [{\bf D}] \not\ni j_0}
&(- \beta)^{|{\bf D}|}
(1+(d_{j_0}-2)\beta)
\prod_{i \in V \backslash [{\bf D}] \atop i \neq j_0}
\lambda_{i}  \nonumber \\
&+
\beta
\sum_{{\bf D} \in \mathcal{D} \atop [{\bf D}] \ni i_0, [{\bf D}] \not\ni j_0}
(- \beta)^{|{\bf D}|}
\prod_{i \in V \backslash [{\bf D}] \atop i \neq j_0}
\lambda_{i}   \\
=:C_1 + \beta C_2.
\end{align*}
In the same manner, $D=D_1 + \beta D_2$.
Similarly,
\begin{align*}
E
&=
\sum_{{\bf D} \in \mathcal{D} \atop [{\bf D}] \not\ni i_0, [{\bf D}] \not \ni j_0}
(- \beta)^{|{\bf D}|}
\lambda_{i_0}\lambda_{j_0}
\prod_{i \in V \backslash [{\bf D}] \atop i \neq i_0,j_0}
\lambda_{i} \\
&=
\sum_{{\bf D} \in \mathcal{D} \atop [{\bf D}] \not\ni i_0, [{\bf D}] \not \ni j_0}
(- \beta)^{|{\bf D}|}
(1+(d_{i_0}-2)\beta)(1+(d_{j_0}-2)\beta)
\prod_{i \in V \backslash [{\bf D}] \atop i \neq i_0,j_0}
\lambda_{i} \nonumber \\
&+
\beta
\sum_{{\bf D} \in \mathcal{D} \atop [{\bf D}] \not\ni i_0, [{\bf D}] \not \ni j_0}
(- \beta)^{|{\bf D}|}
(2+(d_{i_0}+d_{j_0}-3)\beta) 
\prod_{i \in V \backslash [{\bf D}] \atop i \neq i_0,j_0}
\lambda_{i} \\
&=:E_1 + \beta E_2.
\end{align*}
We can straightforwardly check that
\begin{equation} 
\Xi_{G \backslash  e}(- \beta, \bs{\lambda '})
=
B+C_1+D_1+E_1  \nonumber
\end{equation}
and
\begin{equation} 
\beta \Xi_{G / e }(- \beta, \bs{\lambda ''}) 
=
A+\beta C_2 +\beta D_2 + \beta E_2, \label{thmmonomereq10}
\end{equation}
where $\bs{\lambda'}$ and $\bs{\lambda''}$ are defined by the degrees of 
$G \backslash e$ and $G / e$, respectively.
Note that $C_2+D_2$ in Eq.~(\ref{thmmonomereq10}) corresponds to dimer
 arrangements in $G / e$ that cover the new vertex formed by the contraction.
This shows the deletion-contraction relation.
\end{proof}

Let $p_{G}(k)$ be the number of k-matchings of $G$. 
The {\it matching polynomial} $\alpha_{G}$ is defined by
\begin{equation}
\alpha_{G}(x)= \sum_{k=0}^{\lfloor \frac{|V|}{2} \rfloor}(-1)^{k}
p_{G}(k) x^{|V|-2k}.  \nonumber
\end{equation}
The matching polynomial is essentially the monomer-dimer
partition function with uniform weights;
if we set all vertex weights as $\lambda$ and all edge weights as $\mu$, 
we have
\begin{equation}
\Xi_{G}(\mu, \lambda)=
\alpha_{G}\Big( \frac{\lambda}{\sqrt{- \mu}} \Big)
{\sqrt{- \mu}}^{|V|}. \nonumber
\end{equation}
Therefore, for a $(q+1)$-regular graph $G$, 
Theorem \ref{thmmonomer} implies
\begin{equation}
\omega_{G}(u^{2})
=
\alpha_{G}\Big(
\frac{1}{u} +q u \Big)
u^{|V|}.\label{cormatchingeq1}
\end{equation}
In \cite{Nseries}, Nagle derives a sub-coregraph expansion of the
monomer-dimer partition function with uniform weights,
or matching polynomials, 
on regular graphs.
With a transformation of the variables, 
his expansion theorem
is essentially equivalent to
Eq.~(\ref{cormatchingeq1});
Theorem \ref{thmmonomer} gives an extension of the
expansion to non-regular graphs.

As an immediate consequence of Eq.~(\ref{cormatchingeq1}),
we remark on the
symmetry of the coefficients of 
$\omega_{G}$ for regular graphs.
\begin{cor}
Let $G$ be a $(q+1)-$regular graph $(q \geq 1)$ with $N$ vertices and
$w_k$ be the $k$-th coefficient of $\omega_{G}(\beta)$.
Then, we have 
\begin{equation}
\qquad \qquad
w_{N-k}=w_{k} q^{N-2k}
\quad \quad
\text{ for }  0 \leq k \leq N.  \nonumber
\end{equation}
\end{cor}

\subsection{Zeros of $\omega_{G}(\beta)$}
Physicists are interested in
the complex zeros of partition functions, 
because it restricts the occurrence of phase transitions,
i.e., discontinuity of physical quantities with respect to parameters
such as temperature.
In the limit of infinite size of graphs,
the analyticity of the scaled log partition function 
on a complex domain
is guaranteed 
if there exist no zeros in the domain and
some additional conditions hold.
(See \cite{YLstat1,Sbounds}.)
For the monomer-dimer partition
function, 
Heilman and Lieb \cite{HLmonomerdimer}
show the following result.
\begin{thm}
[\cite{HLmonomerdimer} Theorem 4.6] \label{thmmonomerdimerzero}
If $\mu_{e} \geq 0$ for all $e \in E$ and
${\rm Re}(\lambda_{j}) > 0 $  for all $j \in V$,
then
$\Xi_G(\bs{\mu},\bs{\lambda}) \neq 0$.
The same statement is true if 
${\rm Re}(\lambda_{j}) < 0 $  for all $j \in V$.
\end{thm}

Because our polynomial $\omega_{G}(\beta)$  
is a monomer-dimer partition function,
we obtain a bound of the region of complex zeros.

\begin{cor}
Let $G$ be a graph and
let $d_m$ and $d_M$ be the minimum and maximum degree in
$\cg{G}$, respectively,
and assume that $d_m \geq 2$.
If $\beta \in \mathbb{C}$ satisfies $\omega_{G}(\beta)=0$,
then
\begin{equation}
\frac{1}{d_M -1} \leq |\beta| \leq \frac{1}{d_m -1}.  \nonumber
\end{equation}
\end{cor}
\begin{proof}
Without loss of generality, we assume that $G$ is a coregraph.
Let  $\beta = |\beta| {\rm e}^{\img \theta}$ satisfy
$\omega_{G}(\beta)=0$,
where
$0 \leq \theta < 2 \pi$ and $\img$ is the imaginary unit.
Because $\omega_{G}(0)=1$ and the coefficients of $\omega_{G}(\beta)$ are
not negative from Proposition \ref{propomeganonneg},
we have $\beta \neq 0$ and $\theta \neq 0$.
We see that
\begin{equation}
\omega_{G}(\beta)
=
\Xi_G(-\beta,\bs{\lambda}) 
=
\Xi_G(|\beta|,\img {\rm e}^{-\img \theta /2} \bs{\lambda})
(\img{\rm e}^{-\img \theta /2})^{-|V|},  \nonumber
\end{equation}
where $\lambda_j=1+(d_j-1)\beta$
and ${\rm Re}(\img {\rm e}^{-\img \theta /2}\lambda_{j})$
$=(1-(d_j-1)|\beta|) \sin\frac{\theta}{2}$.
From Theorem \ref{thmmonomerdimerzero},
the assertion follows.
\end{proof}
In particular, if the graph is a $(q+1)$-regular graph,
the roots lie on a circle of radius $1/q$,
which is also directly seen by Eq.~(\ref{cormatchingeq1})
by combining the well-known result on the roots of matching polynomials
\cite{HLmonomerdimer}:
the zeros of matching polynomials are on the real interval
$(-2\sqrt{q},2 \sqrt{q})$.

\subsection{Determinant sum formula}
Let 
$\mathcal{T}:=\{C \subset E ; d_{i}(C)=0 \text{ or } 2 \text{ for all }i \in V \}$
be the set of unions of vertex-disjoint cycles.
In this subsection,
an element $C \in \mathcal{T}$ 
is identified with the subgraph $(V_{C},C)$,
where $V_{C}:=\{i \in V ; d_i(C) \neq 0\}$.
A graph $G \setminus C$ is given by deleting all the vertices in $V_{C}$
and the edges of $G$ that are incident with them.

In this subsection, we aim to prove Theorem \ref{coromega},
in which we represent $\omega_{G}$ as a sum of determinants.
This theorem is similar to the expansion of the matching polynomial
by characteristic polynomials \cite{GGmatching}:
\begin{equation}
\alpha_{G}
(x)
=
\sum_{C \in \mathcal{T}}
2^{k(C)}
\det[x I - A_{G \setminus C }], \label{expansionmatching}
\end{equation}
where $A_{G \setminus C }$ is the adjacency matrix of $G \setminus C$
and
$k(C)$, the number of connected components of $C$.

\begin{thm}
\label{coromega}
\begin{equation}
\omega_{G}
(u^{2})
=
\sum_{C \in \mathcal{T}}
2^{k(C)}
\det
\Big(
[
I
-u {A_{G}}
+
u^{2}(D_{G}-I)
]
\Big|_{G\setminus C}
\Big)
u^{|C|}, \label{expansionomega}
\end{equation}
where
$D_G$ is the degree matrix defined by 
$(D_G)_{i,j}:= d_i \delta_{i,j} $ and
$\cdot \big|_{G \setminus C}$ denotes the restriction
to the principal minor indexed by the vertices of $G \setminus C$.
\end{thm}
\begin{proof}
For the proof, we use the result of Chernyak and Chertkov
 \cite{CCfermion2}.
For given weights $\bs{\mu}=(\mu_e)_{e \in E}$ and 
$\bs{\lambda}=(\lambda_i)_{i \in V}$,
a $|V| \times |V|$  matrix $H$ is defined by
\begin{equation}
H:={\rm diag}(\bs{\lambda})- \sum_{e \in E} \sqrt{- \mu_e} A_e, \nonumber
\end{equation}
where $A_{e}=E_{i,j}+E_{j,i}$ for $e=ij$ and $E_{i,j}$ is the matrix base.
In our notation, their result implies
\begin{equation}
\Xi_G(\bs{\mu},\bs{\lambda}) 
=
\sum_{C \in \mathcal{T}}
2^{k(C)}
\det
H |_{G\setminus C}
\prod_{e \in C}\sqrt{- \mu_e}. \nonumber
\end{equation}
If we
set $\lambda_i=1+(d_i - 1)u^2$ and $\sqrt{- \mu_e}=u$,
then the assertion follows.
\end{proof}

For regular graphs, Eqs.~(\ref{expansionmatching}) and
(\ref{expansionomega})
are equivalent
because of Eq.~(\ref{cormatchingeq1}).

The matrix $(I -u {A_{G}} + u^{2}(D_{G}-I))$ is well known 
for its appearance in the Ihara formula of the graph zeta function \cite{STzeta1}.
The result in \cite{WFzeta} shows that the Bethe free energy and the graph zeta function are intimately related
although mathematical relations between the result and Theorem \ref{coromega} are unknown.

\subsection{Values at $\beta=1$}
The value of $\omega_{G}(1)$ is interpreted as the number of a set
constructed from $G$.
For the following theorem, recall that $G^{(2)}$ is
obtained by adding a vertex on each edge in $G=(V,E)$.
The vertices of $G^{(2)}:=(V^{(2)},E^{(2)})$ are classified into 
$V_{O}$ and $V_{A}$,
where $V_{O}$ is the set of original vertices and 
$V_{A}$ is that of newly added ones.
The set of matchings on $G^{(2)}$ is denoted by
$\mathcal{D}_{G^{(2)}}$.
\begin{thm} 
\label{thmomega1}
\begin{equation}
\omega_{G}(1)
=
|\{
{\bf D} \in \mathcal{D}_{G^{(2)}}
;
[{\bf D}]  \supset  V^{}_{O}
\}|.                   \nonumber
\end{equation}
\end{thm}
\begin{proof} 
From Theorem \ref{thmomegaaltrep}, we have
\begin{equation}
\omega_{G}(1)
=
\sum_{s \subset E, s=G_1 \cup \cdots \cup G_{k(s)} 
\atop n(G_j)=1 \text{ for } j=1\ldots k(s)}
 2^{k(s)}, \label{thmomega1eq2}
\end{equation}
where $G_j$ is a connected component of $(V,s)$.
We construct a map $F$ from 
$\{ {\bf D} \in \mathcal{D}_{G^{(2)}}; [{\bf D}] 
\supset  V^{}_{O}  \}$
to $s \subset E$
as 
\begin{equation}
F({\bf D}):= 
\{ e \in E ;
\text{ the half of }e \text{ is covered by an edge in } {\bf D}
\}. \nonumber
\end{equation}
Then, the nullity of each connected component of $F({\bf D})$ is $1$ and
 $|F^{-1}(s)|=2^{k(s)}$.
\end{proof}

\begin{ex}
For the graph $X_3$ in Figure \ref{figomega1}, 
$\omega_{X_3}(1)=\omega_{C_3}(1)=2$.
The corresponding arrangements are also shown in Figure \ref{figomega1}.
\end{ex}

\begin{figure}
\begin{center}
\includegraphics[scale=0.28]{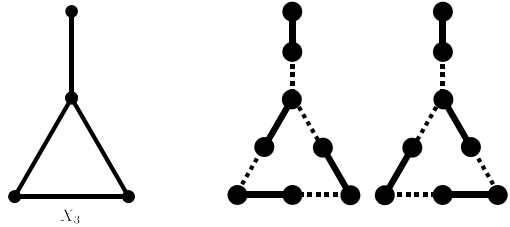}
\vspace{-1mm}
\caption{Graph $X_3$ and possible arrangements on $X^{(2)}_3$.}
\label{figomega1}
\end{center}
\end{figure}

Finally, we remark on the relations between the results on
$\omega_{G}(1)$ obtained in this paper. 
From Proposition \ref{propbasicomega},
$\omega_{G}(1)$ satisfies
\begin{equation}
\omega_{G}(1)=\omega_{G \backslash e}(1)+\omega_{G / e}(1)
\quad 
\text{ if }
e \in E \text{ is not a loop.}  \nonumber
\end{equation}
This relation can be directly observed from the interpretation
of Theorem \ref{thmomega1}.
Theorem \ref{thmmonomer} 
gives
\begin{equation}
 \omega_{G}(1)=
\sum_{{\bf D} \in \mathcal{D}} \hspace{-0.5mm} (-1)^{|{\bf D}|} \hspace{-2mm}
\prod_{i \in V \backslash [{\bf D}]} \hspace{-2mm} d_{i},  \nonumber
\end{equation}
which can be proved from Theorem \ref{thmomega1}
with the inclusion-exclusion principle.
Theorem \ref{coromega} gives
\begin{equation}
\omega_{G}(1)
 =
\sum_{C \in \mathcal{T}} 2^{k(C)}
\det{[D_G-A_G]
\Big|_{G\setminus C}
}. \nonumber
\end{equation}
We can directly prove this formula from Theorem \ref{thmomega1}
using a type of matrix-tree theorem.

\section*{Acknowledgments}
This work was supported in part by Grant-in-Aid for JSPS Fellows
20-993 and Grant-in-Aid for Scientific Research (C) 19500249.

\bibliographystyle{plain}
\bibliography{BP}

\end{document}